\documentclass[11pt]{amsart}

\usepackage{enumerate}
\usepackage{amsmath,amsthm,verbatim,amssymb,amsfonts,amscd,amsopn,amsxtra,graphicx,lmodern,enumitem}
\usepackage{hyperref}
\usepackage{tikz-cd} 
\usepackage{graphics}
\graphicspath{ {./images/} }
\usepackage{mathrsfs}
\usepackage{relsize}
\usepackage{enumitem}
\usepackage[utf8]{inputenc}
\usepackage{csquotes}
\usepackage{amsthm}
\usepackage{thmtools}
\usepackage{mathtools}
\usepackage{microtype}
\usepackage{pdfrender,xcolor}

\tikzset{
	symbol/.style={
		draw=none,
		every to/.append style={
			edge node={node [sloped, allow upside down, auto=false]{$#1$}}}
	}
}

\newlist{condenum}{enumerate}{1} 
\setlist[condenum]{label=\bfseries C\arabic*., 
	ref=\arabic*, wide}

\begin{document}
	\pdfrender{StrokeColor=black,TextRenderingMode=2,LineWidth=0.2pt}	
	
	\title{Common extensions of valuations to rational function fields}

	\author{Arpan Dutta}
	\address{Department of Mathematics, School of Basic Sciences, IIT Bhubaneswar, Argul,
		Odisha, India, 752051.}
	\email{arpandutta@iitbbs.ac.in}
	
	\author{Wael Mahboub}
	\address{Department of Mathematics and Computer Science, Lebanese American University, Beirut Campus, PO Box: 13-5053, Chouran, Beirut 1102-2801, Lebanon.}
	\email{wael.mahboub@lau.edu.lb}

	\def\NZQ{\mathbb}               
	\def\NN{{\NZQ N}}
	\def\QQ{{\NZQ Q}}
	\def\ZZ{{\NZQ Z}}
	\def\RR{{\NZQ R}}
	\def\CC{{\NZQ C}}
	\def\AA{{\NZQ A}}
	\def\BB{{\NZQ B}}
	\def\PP{{\NZQ P}}
	\def\FF{{\NZQ F}}
	\def\GG{{\NZQ G}}
	\def\HH{{\NZQ H}}
	\def\UU{{\NZQ U}}
	\def\P{\mathcal P}
	
	%
	%
	\let\union=\cup
	\let\sect=\cap
	\let\dirsum=\oplus
	\let\tensor=\otimes
	\let\iso=\cong
	\let\Union=\bigcup
	\let\Sect=\bigcap
	\let\Dirsum=\bigoplus
	\let\Tensor=\bigotimes
	
	\theoremstyle{plain}
	\newtheorem{Theorem}{Theorem}[section]
	\newtheorem{Lemma}[Theorem]{Lemma}
	\newtheorem{Corollary}[Theorem]{Corollary}
	\newtheorem{Proposition}[Theorem]{Proposition}
	\newtheorem{Problem}[Theorem]{}
	\newtheorem{Conjecture}[Theorem]{Conjecture}
	\newtheorem{Question}[Theorem]{Question}
	
	\theoremstyle{definition}
	\newtheorem{Example}[Theorem]{Example}
	\newtheorem{Examples}[Theorem]{Examples}
	\newtheorem{Definition}[Theorem]{Definition}
	
	\theoremstyle{remark}
	\newtheorem{Remark}[Theorem]{Remark}
	\newtheorem{Remarks}[Theorem]{Remarks}

	\newcommand{\n}{\par\noindent}
	\newcommand{\nn}{\par\vskip2pt\noindent}
	\newcommand{\sn}{\par\smallskip\noindent}
	\newcommand{\mn}{\par\medskip\noindent}
	\newcommand{\bn}{\par\bigskip\noindent}
	\newcommand{\pars}{\par\smallskip}
	\newcommand{\parm}{\par\medskip}
	\newcommand{\parb}{\par\bigskip}

	\let\epsilon=\varepsilon
	\let\phi=\varphi
	\let\kappa=\varkappa
	
	\newcommand{\trdeg}{\mbox{\rm trdeg}\,}
	\newcommand{\rr}{\mbox{\rm rat rk}\,}
	\newcommand{\sep}{\mathrm{sep}}
	\newcommand{\ac}{\mathrm{ac}}
	\newcommand{\ins}{\mathrm{ins}}
	\newcommand{\res}{\mathrm{res}}
	\newcommand{\Gal}{\mathrm{Gal}\,}
	\newcommand{\ch}{\operatorname{char}}
	\newcommand{\Aut}{\mathrm{Aut}\,}
	\newcommand{\kras}{\mathrm{kras}\,}
	\newcommand{\dist}{\mathrm{dist}\,}
	\newcommand{\ord}{\mathrm{ord}\,}
	\newcommand{\Div}{\mathrm{Div}\,}
	\newcommand{\Supp}{\mathrm{Supp}\,}
	\newcommand{\Spec}{\mathrm{Spec}\,}
	\newcommand{\height}{\mathrm{ht}\,}
	\newcommand{\rk}{\mathrm{rk}\,}
	\newcommand{\Diff}{\mathrm{Diff}\,}
	\newcommand{\Ram}{\mathrm{Ram}\,}
	\newcommand{\id}{\mathrm{id}\,}
	\newcommand{\lex}{\mathrm{lex}\,}

	\let\phi=\varphi
	\let\kappa=\varkappa
	
	\def \a {\alpha}
	\def \b {\beta}
	\def \s {\sigma}
	\def \d {\delta}
	\def \g {\gamma}
	\def \o {\omega}
	\def \l {\lambda}
	\def \th {\theta}
	\def \D {\Delta}
	\def \G {\Gamma}
	\def \O {\Omega}
	\def \L {\Lambda}
	%
	%
	\textwidth=15cm \textheight=22cm \topmargin=0.5cm
	\oddsidemargin=0.5cm \evensidemargin=0.5cm \pagestyle{plain}


	\date{\today}
	
	\maketitle


\begin{abstract}
	Let \( (K(X)|K,w) \) be a valuation transcendental extension of rational function fields and take a minimal pair of definition \((a,\g)\). In this paper, we characterize those \(K\)-conjugates \(a'\) of \(a\) such that the monomial valuation induced by the pair \((a',\g)\) restricts to \(w\) on \(K(X)\). In particular, we show that \(a'\) satisfies this if and only if \(a'\) and \(a\) are conjugates over the henselization of \((K,v)\). 
	
	The second part of the paper concerns abstract key polynomials. We introduce the notion of a regular limit key polynomial, and more generally, that of a regular complete sequence of key polynomials. We prove that regularity is equivalent to the property that every root of every key polynomial determines the corresponding truncated valuation. This extends earlier work of Mahboub, Mansour and Spivakovsky by allowing both limit key polynomials and valuation algebraic extensions. As a consequence, we obtain that \(w\) always admits a regular complete sequence of key polynomials whenever \((K,v)\) is dense in its henselization.    
\end{abstract}


\section{Introduction}
Let $(K,v)$ be a valued field, $\overline{K}$ a fixed algebraic closure of $K$ and $\overline{v}$ an extension of $v$ to $\overline{K}$. Let \(X\) be transcendental over \(K\). The study of extensions of \(v\) to the rational function field \(K(X)\) plays a central role in valuation theory and is closely connected with fundamental problems in algebraic geometry, most notably local uniformization. Two of the principal tools in the investigation of such extensions are the theories of minimal pairs of definition and of key polynomials.

\pars Minimal pairs of definition were introduced in the study of residue-transcendental extensions of valuations to rational function fields. Such extensions (and more generally valuation transcendental extensions) of \(\overline{v}\) to \(\overline{K}(X)\) are given by \textit{monomial} valuations, that is, they are given by Taylor expansion with respect to \(X-a\), where \(a\in\overline{K}\). If the value assigned to \(X-a\) is \(\g\), we denote the corresponding valuation by \(\overline{v}_{a,\g}\). The pair \((a,\g)\) is then called a \textbf{pair of definition} for \(\overline{v}_{a,\g}\). In general, an extension \(\overline{v}_{a,\g}\) may admit several pairs of definition. Alexandru, Popescu and Zaharescu considered the \textit{minimal} objects among them, that is, minimal with respect to the degree of the center \(a\) over \(K\), and showed that they encode significant information regarding the restriction of \(\overline{v}_{a,\g}\) to \(K(X)\). Their work initiated a systematic study of minimal pairs of definition and their applications to extensions of valuations to rational function fields (see \cite{AlexandruPopescuZaharescu1988}, \cite{AlexandruPopescuZaharescu1990a}, \cite{AlexandruPopescuZaharescu1990b}, \cite{AlexandruZaharescu1988}, \cite{Kuhlmann2004BadPlaces}, \cite{Dutta2021}). 

\pars A natural question is the following: Assume that \(w\) is a valuation transcendental extension of \(v\) to \(K(X)\), and let \(\overline{v}_{a,\g}\) be a common extension of \(w\) and \(\overline{v}\) to \(\overline{K}(X)\). Which \(K\)-conjugates \(a'\) of \(a\) satisfy 
\[ \overline{v}_{a',\g}|_{K(X)} = w?  \]
Equivalently, one wants to understand which conjugates of the center \(a\) define the same valuation after restriction to \(K(X)\). This question is substantially more delicate than its analogue over \(\overline{K}(X)\). Indeed, two valuations \(\overline{v}_{a,\g}\) and \(\overline{v}_{a',\g}\) may be distinct on \(\overline{K}(X)\), yet induce the same restriction to \(K(X)\). The following example illustrates this phenomenon.

\begin{Example}
	Let $(K,v)$ denote the valued field $k(t)$ equipped with the $t$-adic valuation. Take an extension $\overline{v}$ of $v$ to $\overline{K}$. Take a real number $\g>\frac{1}{2}$. Then $\overline{v}_{\sqrt{t},\g} \neq \overline{v}_{-\sqrt{t},\g}$. However, for any $z\in K$ we have $\overline{v}_{\sqrt{t},\g}(\sqrt{t}-z) = \min \{ \frac{1}{2},vz \} = \overline{v}_{-\sqrt{t},\g}(-\sqrt{t}-z)$. It now follows from \cite[Theorem 2.2]{AlexandruPopescuZaharescu1990a} that $\overline{v}_{\sqrt{t},\g}|_{K(X)} = \overline{v}_{-\sqrt{t},\g}|_{K(X)}$.
\end{Example}

\pars The first goal of this paper is to answer this question using ramification theoretic techniques. We denote by 
\[ G:= \Gal(\overline{K}|K) = \Gal(\overline{K}(X)|K(X)), \] 
and by \(G^d\) the absolute decomposition group 
\[ G^d:= \left\{ \s\in G \mid \overline{v} \s a = \overline{v} a \text{ for all } a \in \overline{K} \right\}. \] 
We prove that the following are equivalent:
\begin{enumerate}[label=(\roman*)]
	\item \( \overline{v}_{a,\g}|_{K(X)} = \overline{v}_{b,\g}|_{K(X)} \),
	\item \( \overline{v}_{a,\g} = \overline{v}_{\s b, \g} \text{ for some } \s\in G^d \).
\end{enumerate}
In other words, equality after the restriction to \(K(X)\) is controlled by the decomposition group \(G^d\). Using different techniques, this result was also proved in \cite[Proposition 3.12]{PeruginelliNartNovacoski2025}.

\pars After fixing an extension of \( w \) to \( \overline{K(X)} \), we then relate this description to the \textit{henselization} \(K^h\) of \((K,v)\) and to the \textbf{implicit constant field}
\[ IC_K (w) := \overline{K}\sect K(X)^h.   \]
Our first main result of this paper shows that the conjugates of a minimal center which induce the same valuation are determined precisely by these fields.
 
\begin{Theorem}\label{Thm restriction and IC_K(w) and K^h}
	Let \( (K(X)|K,w) \) be a valuation transcendental extension. Fix an extension of \(w\) to \(\overline{K(X)}\). Take a minimal pair of definition \( (a,\g) \) for \( (K(X)|K,w) \) and let \(a'\) be a \(K\)-conjugate of \(a\). Then,
	\begin{enumerate}[label=(\roman*)]
		\item \( \overline{v}_{a,\g} = \overline{v}_{a',\g} \) if and only if \( a' \) is an \( IC_K(w) \)-conjugate of \(a\),
		\item \( \overline{v}_{a,\g}|_{K(X)} = w = \overline{v}_{a',\g}|_{K(X)} \) if and only if \( a' \) is a \( K^h \)-conjugate of \(a\).
	\end{enumerate}
\end{Theorem}
This theorem exhibits a striking parallel between equality of valuations on \(\overline{K}(X)\) and equality after restricting to \(K(X)\): the former is controlled by the implicit constant field \(IC_K(w)\), while the latter is controlled by the henselization \(K^h\). 	 

\pars As a consequence, if the extension \( (K(a)|K,\overline{v}) \) is unibranched, we obtain that 
\[ \overline{v}_{a',\g}|_{K(X)} = w  \]
for \textit{every} \(K\)-conjugate \(a'\). In particular, this conclusion holds whenever \(K\) is dense in the henselization \(K^h\). 

\parm The second part of this paper concerns \textit{key polynomials}. The theory of key polynomials was first introduced by MacLane in the case of discrete rank-one valuations \cite{MacLane1936}, and later extended by Vaqui\'{e} \cite{Vaquie2007Extension} to arbitrary valuations. An alternative approach towards key polynomials was first introduced in \cite{HerreraGovantesMahboubOlallaSpivakovsky2022} and subsequently developed in \cite{NovacoskiSpivakovsky2018}. The corresponding objects are often referred to as \textit{abstract} key polynomials in the literature. Throughout this paper, all key polynomials are understood in this sense.    

\pars The connection between key polynomials and minimal pairs was investigated by Mahboub, Mansour and Spivakovsky in \cite{MahboubMansourSpivakovsky2021}. They showed that if \(w\) is a valuation transcendental extension, and if a complete sequence of key polynomials for \(w\) contains no limit key polynomials, then every root of the last key polynomial defines a common extension of \(w\) and \(\overline{v}\) to \(\overline{K}(X)\). This was also proved by Vaqui\'{e} in the language of admissible families associated to a valuation of \(K[X]\) \cite{Vaquie2005}. 

\pars Our goal is to extend this correspondence to arbitrary extensions of valued fields, allowing both valuation algebraic extensions and the presence of limit key polynomials. The natural induction on the sequence of key polynomials breaks down in the presence of limit key polynomials, because the expected compatibility between consecutive key polynomials need no longer hold. The goal of the second part is to identify the precise additional condition needed to recover the correspondence. For this purpose, we introduce the notions of a \textit{regular limit key polynomial} and of a \textit{regular complete sequence of key polynomials} (see Definitions \ref{Defn regular lim KP} and \ref{Defn regular CSKP}). The condition of regularity happens to be both necessary and sufficient as the next theorem shows:

\begin{Theorem}\label{Thm restriction and regularity}
	Let \((K(X)|K,w)\) be an extension of valued fields. Take a complete sequence of key polynomials \( \left\{Q_i\right\}_{i\in\L} \) for \( (K(X)|K,w) \). For each \( i\in\L \), we denote by \(w_i\) the truncated valuation \(w_{Q_i}\). Then the following are equivalent:
	\begin{enumerate}[label=(\roman*)]
		\item \( \left\{Q_i\right\}_{i\in\L} \) is regular, 
		\item for each \(i\in\L\) and each root \(a\) of \(Q_i\), we have
		\[ \overline{v}_{a,\d(Q_i)}|_{K(X)} = w_i.  \]
	\end{enumerate}
\end{Theorem}
Here \(\d(Q_i)\) denotes the distance of \(X\) from the nearest \(K\)-conjugate of \(a\). As observed in \cite{Novacoski2019}, this value is independent of the choice of the common extension of \(w\) and \(\overline{v}\) to \(\overline{K}(X)\). Thus this theorem shows that regularity is the exact condition in this framework which allows the root-by-root extension property to hold in the presence of limit key polynomials. As an application, we obtain that if \((K,v)\) is dense in the henselization \(K^h\), then every complete sequence of key polynomials is regular. In particular, this holds whenever \( (K,v) \) is henselian or has rank one.


\section{Preliminaries}

Let \((K(X)|K,w)\) be an extension of valued fields. Then \(w\) satisfies the \textbf{Abhyankar inequality}:
\[ \dim_\QQ \left( wK(X)/vK \right) + \trdeg [K(X)w:Kv] \leq 1.  \]
This is a consequence of Theorem 1 of \cite[\S 10.3, Chpater VI]{Bourbaki1989}. If equality holds in the above inequality, then we say that \(w\) is \textbf{valuation transcendental}. Otherwise, \(w\) is said to be \textbf{valuation algebraic}. The property of the extension \(w\) being valuation algebraic or valuation transcendental is preserved under passage from \(K(X)\) to \(\overline{K}(X)\). 


\subsection{Pair of definition} Let \(\overline{w}\) be a common extension of \(w\) and \(\overline{v}\) to \(\overline{K}(X)\). It has been observed in \cite[Theorem 3.11]{Kuhlmann2004BadPlaces} that 
\[  \overline{w} \text{ is valuation transcendental if and only if } \overline{w} = \overline{v}_{a,\g} \text{ for some } a\in \overline{K},    \]
that is, 
\[  \overline{w}(X-z) = \min\left\{ \g, \overline{v}(a-z) \right\} \text{ for all } z\in\overline{K}.  \]
The pair \((a,\g)\) is said to be a pair of definition for \((K(X)|K,w)\). The valuation \(w\) may admit several pairs of definition. Indeed, we have the following from \cite[Proposition 3]{AlexandruZaharescu1988}:
\begin{Proposition}\label{Prop pair of defn}
	Take \( a_1,a_2\in\overline{K} \) and \( \g_1, \g_2 \) in some ordered abelian group containing \(\overline{v}\overline{K} \). Then,
	\[ \overline{v}_{a_1,\g_1} = \overline{v}_{a_2,\g_2} \Longleftrightarrow \g_1 = \g_2 \text{ and } \overline{v}(a_1 - a_2)\geq\g_1.  \]
\end{Proposition}
In light of the above proposition, we define a pair of definition \( (a,\g) \) to be a \textbf{minimal pair of definition for \( (K(X)|K,w) \)} if \(a\) has minimal degree over \(K\) among all pairs of definition, i.e., 
\[ \overline{v}(a-b)\geq\g \Longrightarrow \deg_K(a) \leq \deg_K(b).  \]
The next theorem, proved in \cite[Theorem 2.1]{AlexandruPopescuZaharescu1988}, shows that minimal pairs of definition can be employed to give complete descriptions of the valuation \(w\). Let \(Q(X)\) be the minimal polynomial of \(a\) over \(K\) and take \( f(X)\in K[X] \). Then we have a unique expansion, called the \(Q\)-expansion
\begin{equation}\label{eqn Q-expansion}
	f = \sum_{i=0}^{r} f_i Q^i,
\end{equation}
where \( f_i \in K[X]\) with \( \deg f_i < \deg Q \).

\begin{Theorem}\label{Thm w = v_Q}
	\( wf = \min\{ wf_i + iwQ \} = \min\{ vf_i(a) + iwQ \}.  \)
\end{Theorem}

There exist strong connections between the \(K\)-conjugates of a minimal center and the common extensions of \(w\) and \(\overline{v}\) to \(\overline{K}(X)\). The following result appears in \cite[Theorem 2.2]{AlexandruPopescuZaharescu1990a} and \cite[Theorem 3.11]{Dutta2021}:

\begin{Theorem}\label{Thm common extns APZ}
	Let \(\overline{w_1}\) and \(\overline{w_2}\) be two valuation transcendental extensions of \(\overline{v}\) to \(\overline{K}(X)\). Then the following are equivalent:
	\begin{enumerate}[label=(\roman*)]
		\item \( \overline{w_1}|_{K(X)} = \overline{w_2}|_{K(X)} \),
		\item there exist minimal pairs of definition \((a_i,\g_i)\) for \(\overline{w_i}\) over \(K\) such that 
		\begin{enumerate}
			\item \( \g_1 = \g_2 \),
			\item \(a_1\) and \(a_2\) are conjugates over \(K\), and 
			\[ \overline{v} f(a_1) = \overline{v} f(a_2) \text{ for all } f(X) \in K[X] \text{ with } \deg f < \deg_K(a_1). \]
		\end{enumerate}
	\end{enumerate}
\end{Theorem}

As an immediate consequence, we obtain the following: 

\begin{Proposition}\label{Prop finitely many common extns}
	Let \(w\) be valuation transcendental. Then there exist finitely many common extensions of \(w\) and \(\overline{v}\) to \(\overline{K}(X)\). 
\end{Proposition}


\subsection{The \(j\)-invariant} We recall several facts about the \(j\)-invariant which will be used throughout this paper. The importance of the \(j\)-invariant for the present paper stems from its ability to relate roots of minimal centers to the implicit constant field and truncated valuations. 

\pars Assume that \(w\) is valuation transcendental. Take a common extension \(\overline{w}\) and a minimal pair of definition \((a,\g)\) for \((K(X)|K,w)\). For any polynomial \(f(X)\in K[X]\), we denote by \(\mathcal{R}(f)\) the \textit{multiset} of roots of \(f\). Define 
\[  j_w(f) := \left| \left\{ z\in\mathcal{R}(f) \mid \overline{v}(a-z)\geq\g \right\} \right|. \]  
Observe that we count the roots with multiplicities in the above definition. This operator \( j_w(-) \) was introduced by Dutta in \cite{Dutta2022} and is said to be the \textbf{\(j\)-invariant corresponding to \(w\)}. The invariance of \(j_w(-)\) was established in \cite[Theorem 3.1]{Dutta2024Invariant}. In particular, it is independent of both the chosen extension of \(w\) to \(\overline{K}(X)\) and the chosen pair of definition. 

\pars The \(j\)-invariant encodes valuable ramification theoretic information, as well as providing important descriptions of the valuation \(w\). The following observation is a consequence of \cite[Propositions 3.4 and 2.14]{Dutta2024Invariant} and will play a key role in the sequel: 

\begin{Theorem}\label{Thm j(f)/j(Q)}
	Take a polynomial \( f \in K[X] \). Then,
	\[ j_w(Q) \text{ divides } j_w(f).  \]
	More precisely, considering the expansion (\ref{eqn Q-expansion}), we have
	\[  \dfrac{j_w(f)}{j_w(Q)} = \max \left\{ i \mid wf = wf_i + iwQ \right\}.  \]
\end{Theorem} 

The \(j\)-invariant is also intimately connected to the theory of implicit constant fields which was introduced by Kuhlmann in \cite{Kuhlmann2004BadPlaces}. Fix an extension of \(\overline{w}\) to \(\overline{K(X)}\). Recall that
\[ IC_K(w) := \overline{K}\sect K(X)^h.  \]
It has been observed in \cite[Lemma 5.1]{Dutta2021} that 
\[  IC_K(w) \subseteq K(a)^h. \]
Since \(IC_K(w)\) is relatively algebraically closed in \(K(X)^h\), the algebraic extension \(K(a)^h|IC_K(w)\) is linearly disjoint to \(K(X)^h|IC_K(w)\). As a consequence, we conclude the following from \cite[Proposition 2.9]{DuttaGhosh2025}:

\begin{Proposition}\label{Prop j(Q) = [K(a)^h:IC_K(w)]}
	\( [K(a)^h : IC_K(w)] = j_w(Q) \).
\end{Proposition} 


\subsection{Common extensions and the \(j\)-invariant} Assume that \(w\) is valuation transcendental. By Proposition \ref{Prop finitely many common extns}, there exist finitely many common extensions of \(w\) and \(\overline{v}\) to \(\overline{K}(X)\), say
\[ w^{(1)}, \dotsb, w^{(t)}. \]
For each \(w^{(i)}\), fix a minimal pair of definition \( (a_i,\g) \). We can choose the \(a_i\) to be conjugates over \(K\) by Theorem \ref{Thm common extns APZ}. For any polynomial \(f(X) \in K[X]\), define the multisets
\begin{align*}
	\mathcal{S}^{(i)}_w (f) &:= \left\{ z\in\mathcal{R}(f) \mid \overline{v}(a_i - z)\geq \g \right\},\\
	\mathcal{S}_w (f) &:= \left\{ z \in \mathcal{R}(f) \mid \overline{v}_{z,\g}|_{K(X)} = w \right\}. 
\end{align*}
Observe that
\[ i \neq j \Longrightarrow  \mathcal{S}^{(i)}_w (f) \sect \mathcal{S}^{(j)}_w (f) = \emptyset. \]
As a consequence, 
\[ \mathcal{S}_w (f) = \bigsqcup_i \mathcal{S}^{(i)}_w (f)  \]
is a disjoint union. Moreover, since the operator \( j_w(-) \) is independent of the choice of the common extension, we have
\[  \left| \mathcal{S}^{(i)}_w (f) \right| = j_w(f) \text{ for all } i. \]
We have thus arrived at the following result which will play a crucial role in the proof of Theorem \ref{Thm restriction and regularity}:

\begin{Theorem}\label{Thm S_f and j_w(f)}
	Assume that \( (K(X)|K,w) \) is valuation transcendental. Let \( w^{(1)}, \dotsb, w^{(t)} \) be all the common extensions of \(w\) and \(\overline{v}\) to \(\overline{K}(X)\). Then for any polynomial \(f(X) \in K[X]\), we have 
	\[ \left| \mathcal{S}_w (f) \right| = t \,  j_w(f).   \]
\end{Theorem}

\subsection{Key polynomials} We now assume that \(w\) is an \textit{arbitrary} extension of \(v\) to \(K(X)\), and take a common extension \(\overline{w}\) of \(w\) and \(\overline{v}\) to \(\overline{K}(X)\). For any polynomial \(f(X)\in K[X]\), define 
\[ \d_w(f):= \max \left\{ \overline{w}(X-z) \mid z\in \mathcal{R}(f) \right\}. \]
It has been observed in \cite[Remark 3.2]{Novacoski2019} that the value \(\d_w(f)\) is independent of the choice of the extension of \(w\).

\begin{Definition}
	A \textit{monic} polynomial \(Q(X) \in K[X]\) is said to be a \textbf{key polynomial} for \((K(X)|K,w)\) if for any polynomial \(f(X) \in K[X]\), we have 
	\[ \deg f < \deg Q \Longrightarrow \d_w(f) < \d_w(Q).  \]
\end{Definition}
Thus key polynomials are irreducible over \(K\). Moreover, every monic linear polynomial is a key polynomial. 

\pars Given any monic irreducible polynomial \(Q\) over \(K\), we define the \textbf{\(Q\)-truncation} of \(w\) as follows: for any polynomial \(f\) over \(K\), consider its \(Q\)-expansion (\ref{eqn Q-expansion}). Then,
\[ w_Q f:= \min \left\{ w(f_iQ^i) \right\}.  \] 
A sufficient condition for \(w_Q\) to be a valuation on \(K(X)\) is that \(Q\) be a key polynomial for \((K(X)|K,w)\) \cite[Proposition 2.6]{NovacoskiSpivakovsky2018}, but it is not a necessary condition \cite[Proposition 2.3]{Novacoski2019}. Employing this notion of \(Q\)-truncation, the connection between minimal pairs and key polynomials has been explored in \cite[Theorem 1.1]{Novacoski2019}: 

\begin{Theorem}\label{Thm key pols min pairs}
	Take an extension \(\overline{w}\) of \(w\) to \(\overline{K}(X)\). Let \(Q(X) \in K[X]\) be a monic irreducible polynomial and choose a root \(a\in\mathcal{R}(Q)\) such that \(\overline{w}(X-a) = \d_w(Q)\). Then the following are equivalent:
	\begin{enumerate}[label=(\roman*)]
		\item \(Q\) is a key polynomial for \((K(X)|K,w)\),
		\item \( (a,\d_w(Q)) \) is a minimal pair of definition for \( (K(X)|K, w_Q) \).
	\end{enumerate}
\end{Theorem}

The following result gives a criterion for the value of a polynomial to coincide with its truncated value:

\begin{Proposition}\cite[Proposition 2.4]{Dutta2023MathNach}\label{Prop wf = w_Q f}
	Let \(Q\) be a key polynomial for \((K(X)|K,w)\). Then for any polynomial \( f \in K[X] \),
	\[  w f  = w_Q f \text{ if and only if } \d_w(f) \leq \d_w (Q).  \]
\end{Proposition}


\subsection{Complete sequence of key polynomials}

\begin{Definition}
	A family \(\left\{Q_i\right\}_{i\in\L}\) of key polynomials for \((K(X)|K,w)\) is said to form a \textbf{complete sequence of key polynomials} if it satisfies the following properties: 
	\sn (CSKP1) \( \d_w(Q_i) \neq \d_w(Q_{i'}) \) for all \( i, i' \in \L \) with \(i\neq i'\), 
	\n (CSKP2) \(\L\) is well ordered with respect to the ordering given by \(i < i'\) if and only if \(\d_w(Q_i) < \d_w(Q_{i'})\), 
	\n (CSKP3) for any \(f(X)\in K[X]\), there exists some \(i\in\L\) such that \(\deg Q_i \leq \deg f\) and \(wf = w_{Q_i} f\). 
\end{Definition}

\begin{Remark}\label{Rmk CSKP}
	It has been observed in \cite[Theorem 1.1]{NovacoskiSpivakovsky2018} that every valued field extension \( (K(X)|K,w) \) admits a complete sequence of key polynomials. Moreover, we can assume that a complete sequence \(\left\{Q_i\right\}_{i\in\L}\) satisfies the following properties: \smallskip
	\begin{enumerate}
		\item \( \L = \Union_{j\in I} \L_j \) where \(I\) is an initial segment of \(\NN\). \smallskip
		\item For each \(j\in I\), we have \( \L_j = \{j\} \union \vartheta_j \), where \(\vartheta_j\) is an ordered set without a maximal element, which may also be empty. \smallskip
		\item \(Q_0 = X\).\smallskip 
		\item For each \( j\in I\setminus \{0\} \) and \( i\in \vartheta_{j-1} \), we have \( j-1 < i < j \).\smallskip
		\item \( \deg Q_i = \deg Q_{i'} \) for all \( i, i' \in \L_j \). \smallskip
		\item \( \deg Q_i < \deg Q_{i'} \) for all \( i\in\L_j \) and \( i'\in\L_{j+1} \).\smallskip
		\item \( w Q_i < w Q_{i'} \) and \( \d_w(Q_i) < \d_w(Q_{i'}) \) for all \( i < i' \in \L \).
	\end{enumerate}
\end{Remark}

Even though the set of key polynomials \(\left\{Q_i\right\}_{i\in\L}\) is not unique, the cardinality of the indexing set \(I\), the degrees of the key polynomials \(\deg Q_i\) and the truncated valuations \(w_{Q_i}\) are uniquely determined by \(w\). 

\begin{Definition}
	In the framework of Remark \ref{Rmk CSKP}, we say that \(Q_j\) is a \textbf{limit key polynomial} for \((K(X)|K,w)\) if the following conditions are satisfied:
	\[ j \in I\setminus\{0\} \text{ and } \vartheta_{j-1} \neq \emptyset.  \]
\end{Definition}

\begin{Remark}\label{Rmk regular non-limit key pols}
	If \( j\in\L \) is a non-maximal index such that \( Q_{j+1} \) is not a limit key polynomial (that is, \( \vartheta_j = \emptyset \)), then it follows from \cite[Theorem 26]{DecaupMahboubSpivakovsky2018} that \( Q_{j+1} \) is a key polynomial for \( (K(X)|K, w_{Q_j}) \) in the sense of MacLane-Vaqui\'{e}. Therefore, upon considering the \(Q_j\)-expansion of \(Q_{j+1}\),
	\[ Q_{j+1} = q_n Q_j^n + \dotsc + q_0,  \]
	we conclude the following from \cite[Remark 3.8]{Dutta2024Invariant}:
	\begin{align*}
		q_n &= 1, \\
		w_{Q_j} (Q_{j+1}) &= n w Q_j \leq w q_i + i w Q_j \text{ for all } i < n. 
	\end{align*}
\end{Remark}

This motivates the following definitions: 

\begin{Definition}\label{Defn regular lim KP}
	Let \(j\in\L\) be such that \(Q_{j+1}\) is a limit key polynomial. We say that \(Q_{j+1}\) is a \textbf{regular limit key polynomial} if there exists \( i_0 \in \vartheta_j \) such that for all \( i \geq i_0 \), the \(Q_i\)-expansion of \(Q_{j+1}\) has the form 
	\[ Q_{j+1} = Q_i^n + q_{n-1} Q_i^{n-1} + \dotsc + q_1 Q_i + q_0,   \]
	and 
	\[ w_{Q_i}(Q_{j+1}) = n w Q_i \leq w q_t + t w Q_i \text{ for all } t < n.  \]
\end{Definition}

\begin{Definition}\label{Defn regular CSKP}
	The complete set of key polynomials \( \{Q_j\}_{j\in\L} \) is said to be \textbf{regular} if every non-limit key polynomial satisfies the conclusions of Remark \ref{Rmk regular non-limit key pols} and if every limit key polynomial is regular.
\end{Definition}

\begin{Remark}
	Both regular and non-regular limit key polynomials may exist. See \cite[Example 5.2.1]{Mahboub2013} for an example when the limit key polynomial is regular, and \cite[Example 5.3.1]{Mahboub2013} for a complementary example.
\end{Remark}

\begin{Remark}
	We now fix some notations that will be used in the sequel. Consider the setup of Remark \ref{Rmk CSKP}. For a non-maximal index \(i\in \L\), we define
	\[ i^+ := \begin{cases}
		i + 1, &\quad \text{ if \( i \in I \) and \( i + 1 \in I \)},\\
		j + 1, &\quad \text{ if \( i \in \vartheta_j \) and \( j + 1 \in I \)}.
	\end{cases}  \]
Observe that \(Q_{i^+}\) is a limit key polynomial in the latter scenario. 

\pars For \( i, i' \in \L \), we write \( i \equiv i' \) if \( \deg Q_i = \deg Q_{i'} \). Thus one of the following conditions hold:
\begin{align*}
	i \in I      &\text{ and } i' \in \vartheta_i,\\
	i' \in I     &\text{ and } i \in \vartheta_{i'},\\
	\text{there exists } j\in I &\text{ such that } i,i' \in \vartheta_j.
\end{align*}
For \( i < i' \in \L \), we will write 
\[ i' \sim i \text{ if } i' = i^+ \text{ or } i' \equiv i.   \]
\end{Remark}


\section{Proof of Theorem \ref{Thm restriction and IC_K(w) and K^h}}

\begin{Proposition}\label{Prop restriction iff a, gamma = sigma b, gamma}
	Take $a,b \in \overline{K}$ and $\g$ in some ordered abelian group containing $\overline{v}\overline{K}$. Then the following are equivalent:
	\sn (i) $\overline{v}_{a,\g}|_{K(X)} = \overline{v}_{b,\g}|_{K(X)}$,
	\sn (ii) $\overline{v}_{a,\g} = \overline{v}_{\s b, \g}$ for some $\s\in G^d$.
\end{Proposition}

\begin{proof}
	If \( \g < \overline{v}z \) for all \( z\in\overline{K} \), then \( \overline{v}_{a,\g}  = \overline{v}_{b,\g} \) for all \( a,b \in \overline{K} \) by Proposition \ref{Prop pair of defn}. Hence both (i) and (ii) hold trivially. We now consider the complementary case. In particular, we can assume that
	\begin{equation}\label{eqn gamma' < gamma}
		\text{there exists } \g'\in\overline{v}\overline{K} \text{ such that } \g' < \g.
	\end{equation}
	
	\pars We first show the direction $(ii)\Longrightarrow (i)$. Employing the fact that $\s\in G^d$, we observe that for any $z\in\overline{K}$, we have
	\begin{align*}
		\overline{v}_{\s b,\g}(X-z) &= \min\{ \g, \overline{v}(\s b-z) \}\\
		&= \min\{ \g, \overline{v}(b - \s^{-1}z) \}\\
		&= \overline{v}_{b,\g} (X- \s^{-1}z) \\
		&= (\overline{v}_{b,\g} \circ \s^{-1}) (X-z). 
	\end{align*}
Thus 
\[ \overline{v}_{\s b,\g} = \overline{v}_{b,\g} \circ \s^{-1}. \]
Since $\s^{-1}\in \Gal(\overline{K}(X)|K(X))$, we observe that $\overline{v}_{\s b , \g}$ and $\overline{v}_{b,\g}$ are conjugate valuations and hence
\[  \overline{v}_{a,\g}|_{K(X)} = \overline{v}_{\s b, \g}|_{K(X)} = \overline{v}_{b,\g}|_{K(X)}. \]

\pars We now show $(i)\Longrightarrow (ii)$. The condition $\overline{v}_{a,\g}|_{K(X)} = \overline{v}_{b,\g}|_{K(X)}$ implies that $\overline{v}_{a,\g} = \overline{v}_{b,\g}\circ \tau$ for some $\tau\in G$. Then
\[ \overline{v}_{a,\g}(X-\tau^{-1}b) = \overline{v}_{b,\g}(X-b) = \g, \]
whence $\overline{v}(a-\tau^{-1}b)\geq\g$. It follows that
\[ \overline{v}_{\tau^{-1}b,\g}= \overline{v}_{a,\g} = \overline{v}_{b,\g}\circ \tau. \]
As a consequence, we obtain that
\begin{equation*}
	\min\{ \g, \overline{v} (\tau^{-1}b-z) \} = \min\{ \g, \overline{v}(b -\tau z) \} \text{ for all } z\in\overline{K}.
\end{equation*}
Since $\tau$ and $b$ are fixed, $\tau^{-1}b-z$ runs over entire $\overline{K}$ as $z$ runs over all $\overline{K}$. The above expression can thus be rewritten as 
\begin{equation*}
	\min\{ \g, \overline{v}z \} = \min\{ \g, \overline{v}\tau z \} \text{ for all } z\in\overline{K}.
\end{equation*}
Therefore,
\begin{equation}\label{eqn vz = v tau z}
	\overline{v}z < \g \Longrightarrow \overline{v} z = \overline{v} \tau z.
\end{equation}
If $\g > \overline{v}z$ for all \( z \in \overline{K} \), then this implies that $\overline{v}z = \overline{v}\tau z$ for all $z\in\overline{K}$, whence $\tau\in G^d$. We now consider the remaining case. Suppose if possible that $\tau\notin G^d$. Then there exists $z_1\in\overline{K}$ such that $\overline{v}z_1 \neq \overline{v}\tau z_1$. In light of (\ref{eqn gamma' < gamma}), we can choose \( z_2 \in \overline{K} \) such that 
\[ \overline{v} z_2 < \g \text{ and } \overline{v} (z_1 z_2) < \g.  \]
From (\ref{eqn vz = v tau z}) and our choice of \(z_1\), we then obtain
\[ \overline{v} (z_1 z_2) \neq \overline{v} \tau (z_1 z_2).  \]
However, this contradicts (\ref{eqn vz = v tau z}). It follows that $\tau\in G^d$. Setting $\s = \tau^{-1}$ we have the assertion. 
\end{proof}

\pars For the remainder of this section, we assume that \((K(X)|K,w)\) is a valuation transcendental extension. Let \( \overline{w} \) be a common extension of \(w\) and \(\overline{v}\) to \(\overline{K}(X)\). We fix a minimal pair of definition $(a,\g)$ for $(K(X)|K,w)$ and an extension of $\overline{w}$ to $\overline{K(X)}$. Thus we can talk of the implicit constant field
\[ IC_K(w) := \overline{K}\sect K(X)^h. \] 
Set 
\[  G^d_X:= G^d (\overline{K(X)}|K(X),\overline{w}) = \{ \s\in \Gal(\overline{K(X)}|K(X)) \mid \overline{w}\circ\s = \overline{w} \text{ on } \overline{K(X)} \}.\]
Then $IC_K(w)$ is the fixed field of the subgroup $G_{\res}$ of $G$, where
\begin{equation}\label{eqn G_res and G_d^X}
	G_{\res}:= \left\{ \s|_{K^\sep} \mid \s\in G^d_X \right\}.
\end{equation}
Observe that $K^h$ is a subfield of $IC(K(X)|K,w)$. As a consequence, 
\[ G_{\res} \subseteq G^d. \]
We now provide a \textit{proof of Theorem \ref{Thm restriction and IC_K(w) and K^h}}.

\begin{proof}
	Consider the multisets:
	\begin{align*}
		\mathcal{S} &:= \left\{ K \text{-conjugates \(a'\) of \(a\) such that } \overline{v}(a-a') \geq \g  \right\},\\
		\mathcal{I} &:= \left\{ IC_K(w) \text{-conjugates of } a  \right\},
	\end{align*}
where conjugates are counted with multiplicities. Take \( a' \in \mathcal{I} \). Then \( a' = \s a \) for some \( \s\in G_{\res} \). It follows from (\ref{eqn G_res and G_d^X}) that 
\[ \overline{w} (X-a') = \left( \overline{w}\circ\s \right)(X-a) = \overline{w}(X-a) = \g.  \]
As a consequence, \( \overline{v}(a-a') \geq \g \) and hence 
\[ \mathcal{I} \text{ is a submultiset of } \mathcal{S}. \]
Observe that 
\[ \left|\mathcal{S}\right| = j_w(Q) \]
where \(Q\) is the minimal polynomial of \(a\) over \(K\). On the other hand, since \(IC_K(w) \subseteq K(a)^h\) by \cite[Lemma 5.1]{Dutta2021}, we have 
\[ \left|\mathcal{I}\right| = \left[K(a)^h:IC_K(w)\right]. \]
It now follows from Proposition \ref{Prop j(Q) = [K(a)^h:IC_K(w)]} that the cardinalities of the multisets \(\mathcal{S}\) and \(\mathcal{I}\) coincide. Since \( \mathcal{I}\) is a submultiset of \( \mathcal{S} \), we conclude that
\[ \mathcal{I} = \mathcal{S}.  \]
We have thus proved the first assertion. 

\pars We now prove the second assertion. First assume that $w = \overline{v}_{a,\g}|_{K(X)} = \overline{v}_{a',\g}|_{K(X)}$. It follows from Proposition \ref{Prop restriction iff a, gamma = sigma b, gamma} that $\overline{v}(a-\s a') \geq\g$ for some $\s\in G^d$, whence $\s a' = \tau a$ for some $\tau\in G_\res$ by part (i). Since $G_\res$ is a subgroup of $G^d$, we conclude that $a'$ is a $K^h$-conjugate of $a$. The reverse direction is a direct application of Proposition \ref{Prop restriction iff a, gamma = sigma b, gamma}.
\end{proof}

\pars If $a\in\overline{K}$ is such that the extension $(K(a)|K,\overline{v})$ is unibranched, then it has been observed in \cite[Lemma 2.1]{BlaszczokKuhlmann2017} that $[K(a):K] = [K^h(a):K^h]$. As a consequence, every $K$-conjugate of $a$ is also a $K^h$-conjugate of $a$. The following is now immediate from Theorem \ref{Thm restriction and IC_K(w) and K^h}:

\begin{Corollary}\label{Coro res = w unibranched}
	Assume that the extension $(K(a)|K,\overline{v})$ is unibranched. Then 
	\[ \overline{v}_{a',\g}|_{K(X)} = w \text{ for all conjugates } a' \text{ of $a$ over }K.  \] 
	In particular, this assertion holds for all $a\in\overline{K}$ whenever $(K,v)$ is henselian.
\end{Corollary}

\pars A particularly important case of Corollary \ref{Coro res = w unibranched} is mentioned underneath. We first make a definition: a valued field $(K,v)$ is said to be \textbf{dense} in an extension $(K', v')$ if for any $\d\in v' K'$ and $a'\in K'$, there exists $a\in K$ such that $v'(a-a') > \d$. 

\begin{Corollary}\label{Coro res = w dense in K^h}
	Assume that $(K,v)$ is dense in the henselization $K^h$. Then
	\[ \overline{v}_{a',\g}|_{K(X)} = w \text{ for all conjugates $a'$ of $a$ over }K.   \]
	In particular, this assertion holds whenever \((K,v)\) has rank one.
\end{Corollary}

\begin{proof}
	$(K,v)$ being dense in $K^h$ implies that $K^h$ is a subset of the completion $\widehat{K}$ of $(K,v)$. It then follows from \cite[Corollary 3.2]{Dutta2023MathNach} that $[K(a):K] = [K^h(a):K^h]$, whence the extension $(K(a)|K,\overline{v})$ is unibranched by \cite[Lemma 2.1]{BlaszczokKuhlmann2017}. The result now follows from Corollary \ref{Coro res = w unibranched}.
\end{proof}


\section{Proof of Theorem \ref{Thm restriction and regularity}}

Let \( (K(X)|K,w) \) be an \textit{arbitrary} extension of valued fields. We fix a common extension \( \overline{w} \) of \(w\) and \(\overline{v}\) to \(\overline{K}(X)\). By \cite[Theorem 1.1]{NovacoskiSpivakovsky2018}, fix a complete sequence of key polynomials \( \left\{Q_i\right\}_{i\in\L} \) satisfying the assumptions of Remark \ref{Rmk CSKP}. Furthermore, we can assume that every non-limit key polynomial satisfies the conclusions of Remark \ref{Rmk regular non-limit key pols}. For all \(i\in\L\), we fix the following notations to be used throughout the rest of the section:
\[  w_i:= w_{Q_i}, \, \d_i:= \d_w(Q_i).  \]

\pars The following result is an extension of \cite[Proposition 6.5]{MahboubMansourSpivakovsky2021} to allow for limit key polynomials. 

\begin{Proposition}\label{Prop MMS}
	Let \( i\in \L \) be non-maximal and \(i' \sim i\). Take a root \( a\in \mathcal{R}(Q_{i'}) \). If \( i' = i+1 \) or \( i' \equiv i \), assume that there exists \( b \in \mathcal{R}(Q_i) \) such that 
	\[ \overline{v}_{b,\d_i}|_{K(X)} = w_i \text{ and } \overline{v}(a-b) \geq \d_i. \]
	If \( i' = j+1 \) for some \(j\in I\) where \(i\in\vartheta_j\), assume that there exists \( i_0 \in \vartheta_j \) such that for all \( k\in\vartheta_j \) with \(k\geq i_0\), there exist roots \( b_k \in \mathcal{R}(Q_k) \) satisfying 
	\[ \overline{v}_{b_k,\d_k}|_{K(X)} = w_k \text{ and } \overline{v}(a-b_k) \geq \d_k. \]
	Then,
	\[ \overline{v}_{a,\d_{i'}}|_{K(X)} = w_{i'}. \]
\end{Proposition} 

\begin{proof}
	The first assertion has been proved in \cite[Proposition 6.5]{MahboubMansourSpivakovsky2021}. We now consider the remaining case where \(Q_{j+1}\) is a limit key polynomial. By Theorem \ref{Thm key pols min pairs}, there exists a root \( a' \in \mathcal{R}(Q_{j+1}) \) such that \( (a', \d_{j+1}) \) is a minimal pair of definition for \( (K(X)|K,w_{j+1}) \). In light of Theorem \ref{Thm common extns APZ}, it now suffices to show that
	\begin{equation*}
		\overline{v} g(a') = \overline{v} g(a) \text{ for all } g\in K[X] \text{ with } \deg g < \deg Q_{j+1}. 
	\end{equation*}
	Observe that \( \overline{v} g(a') = w_{j+1}g \) by Theorem \ref{Thm w = v_Q}. Moreover, since \(Q_{j+1}\) is a key polynomial and \( \deg g < \deg Q_{j+1} \), we have \( \d_w(g) < \d_{j+1} \). As a consequence, Proposition \ref{Prop wf = w_Q f} yields \( w_{j+1}g = wg \). Therefore the above expression can be rewritten as 
	\begin{equation}\label{eqn wg = vg(a)}
		wg = \overline{v} g(a) \text{ for all } g\in K[X] \text{ with } \deg g < \deg Q_{j+1}.
	\end{equation} 
	\pars Take some polynomial \( g\in K[X] \). By the condition (CSKP3), we can choose \(k\in\vartheta_j\) such that \( \d_w(g) \leq \d_k \). Hence \( wg = w_k g \) by Proposition \ref{Prop wf = w_Q f}. From the given conditions, we can further assume that \( \overline{v}(a-b_k) \geq \d_k \), whereby \( (a,\d_k) \) is a pair of definition for \(w_k\) by Proposition \ref{Prop pair of defn}. Thus
	\[  \overline{v}_{a,\d_k} (X-z) \leq \overline{v}(a-z) \text{ for all } z\in\overline{K}. \]
	It follows that
	\[ w_k f \leq \overline{v} f(a) \text{ for all } f(X)\in K[X]. \]
	In particular, we obtain
	\[ wg = w_k g \leq \overline{v} g(a). \]
	Suppose that the above inequality is strict. Furthermore, assume that \( \deg g \) is minimal with respect to this property. Consider the Euclidean division
	\[ Q_{j+1} = hg + r, \text{ where } \deg r < \deg g. \]
	The minimality of \(\deg g\)  implies that
	\[ wr = \overline{v} r(a). \]
	Since \(a\in \mathcal{R}(Q_{j+1})\), we have \( h(a)g(a) + r(a)= 0 \). Consequently, we obtain
	\[ wr = \overline{v} r(a) = \overline{v} (h(a) g(a)) > wh + wg = w (hg),  \]
	where the inequality follows from the assumption on \(g\). The triangle inequality now yields
	\[  w Q_{j+1} = w(hg).  \]
	Since \( \deg h, \, \deg g, \, \deg r < \deg Q_{j+1} \), we can further assume by condition (CSKP3) that
	\[ \d_w (hg) = \max \left\{  \d_w(h), \d_w(g)  \right\}  \leq \d_k \text{ and } \d_w(r) \leq \d_k. \]
	Hence Proposition (\ref{Prop wf = w_Q f}) yields 
	\[ w_k (hg) = w (hg) \text{ and } w_k r = wr. \]
	We have thus arrived at
	\[ w Q_{j+1} = w_k Q_{j+1}.  \]
	Since \( \d_{j+1} > \d_k \), this contradicts Proposition \ref{Prop wf = w_Q f}. Thus (\ref{eqn wg = vg(a)}) is satisfied for all polynomials \(g\in K[X]\) with \(\deg g < \deg Q_{j+1}\).
\end{proof}

\parm We are now well-posed to provide a \textit{proof of Theorem \ref{Thm restriction and regularity}}. Recall the notations of Theorem \ref{Thm S_f and j_w(f)}.

\begin{proof}
	For each \( i\in\L \), denote by \(t_i\) the number of common extensions of \(w_i\) and \(\overline{v}\) to \(\overline{K}(X)\). Observe that \(t_i\) is a finite number by Proposition \ref{Prop finitely many common extns}. 
	
	\pars We first prove the direction \( (i) \Longrightarrow (ii) \) by transfinite induction on \(i\). The assertion is vacuously true for \(i = 0\). Assume that the claim holds for some non-maximal element \(i\in\L\). We will show that it holds for all \(i'>i\) with \(i' \sim i\). First assume that either \(i' = i+1\) or \(i'\equiv i\). Consider the \(Q_i\)-expansion 
	\[ Q_{i'} = Q_i^n + \dotsc + q_1 Q_i + q_0. \]
	Since \(\{ Q_i \}_{i\in\L}\) is regular, employing Theorem \ref{Thm j(f)/j(Q)} we obtain 
	\[ j_{w_i} (Q_{i'}) = n j_{w_i} (Q_i). \]
	We have assumed that 
	\[ \overline{v}_{b,\d_i}|_{K(X)} = w_i \text{ for all } b\in \mathcal{R}(Q_i).  \]
	In other words, 
	\[ \left| \mathcal{S}_{w_i} (Q_i) \right| = \deg Q_i.  \]
	Applying Theorem \ref{Thm S_f and j_w(f)}, we have
	\begin{align*}
		t_i \, j_{w_i} (Q_i) &= \deg Q_i,\\
		\left| \mathcal{S}_{w_i} (Q_{i'}) \right| = t_i \, j_{w_i} (Q_{i'}) &= n \, t_i \, j_{w_i} (Q_i) = n \deg Q_i = \deg Q_{i'}.
	\end{align*}
	In other words, for any \( a\in\mathcal{R}(Q_{i'}) \), we have
	\[ \overline{v}_{a,\d_i}|_{K(X)} = w_i.  \]
	By Theorem \ref{Thm key pols min pairs}, there exists \( b\in\mathcal{R}(Q_i) \) such that \( \overline{v}_{a,\d_i} = \overline{v}_{b,\d_i} \). As a consequence,
	\begin{equation}\label{eqn *}
		\overline{v}_{b,\d_i}|_{K(X)} = w_i \text{ and } \overline{v}(a-b) \geq \d_i.
	\end{equation}
	Then
	\[  \overline{v}_{a,\d_{i'}}|_{K(X)} = w_{i'}  \]
	by Proposition \ref{Prop MMS}.
	
	\pars We now assume that \(i' = j+1\) for some \(j\in I\) with \(i\in\vartheta_j\). Since \(\{Q_i\}_{i\in\L}\) is regular, there exists some \(i_0 \in \vartheta_j\) such that for all \( k \geq i_0 \) with \(k\in\vartheta_j\), the \(Q_k\)-expansion
	\[ Q_{j+1} = q_n Q_k^n + \dotsc + q_1 Q_k + q_0  \]
	satisfies 
	\[ q_n = 1 \text{ and } w_k (Q_{j+1}) = n w Q_k.   \]
	Hence 
	\[ j_{w_k} (Q_{j+1}) = n j_{w_k} (Q_k)  \]
	by Theorem \ref{Thm j(f)/j(Q)}. Moreover, since \( k\equiv i \), we obtain from the first part of the induction argument that the assertion holds true for \(k\). Hence,
	\[ t_k \, j_{w_k} (Q_k) = \left| \mathcal{S}_{w_k} (Q_k) \right| = \deg Q_k.    \]
	It now follows from Theorem \ref{Thm S_f and j_w(f)} that 
	\[ \left| \mathcal{S}_{w_k} (Q_{j+1}) \right| = t_k \, j_{w_k} (Q_{j+1}) = t_k \, n j_{w_k} (Q_k) = n \deg Q_k = \deg Q_{j+1}.   \]
	Analogous to (\ref{eqn *}), we deduce that 
	\[ \overline{v}_{b_k, \d_k}|_{K(X)} = w_k \text{ and } \overline{v}(a-b_k) \geq \d_k. \]
	Since this holds for all \( k\in\vartheta_j \) with \(k\geq i_0\), the hypotheses of Proposition \ref{Prop MMS} are satisfied, and hence
	\[ \overline{v}_{a,\d_{j+1}}|_{K(X)} = w_{j+1}.  \]
	 
	 \parm We now prove the direction \( (ii) \Longrightarrow (i) \). By Remark \ref{Rmk regular non-limit key pols}, it suffices to show that every limit key polynomial is regular. Take \( j\in I \) such that \(Q_{j+1}\) is a limit key polynomial. Suppose that \(Q_{j+1}\) is not regular. Then there exists \(i\in\vartheta_j\) such that in the \(Q_i\)-expansion 
	 \[ Q_{j+1} = q_n Q_i^n + \dotsc + q_1 Q_i + q_0, \]
	 either \(q_n \neq 1\) or \( w_i(Q_{j+1}) < n w Q_i \). Since \(Q_i\) and \(Q_{j+1}\) are monic, the criterion \(q_n\neq 1\) forces \(q_n\) to be a monic polynomial of positive degree. Consequently,
	 \[ n \deg Q_i < \deg Q_{j+1}.  \]
	 If \( q_n = 1 \) and \( w_i(Q_{j+1}) < n w Q_i \), then Theorem \ref{Thm j(f)/j(Q)} would yield that
	 \[ j_{w_i} (Q_{j+1}) < n j_{w_i} (Q_i).  \]
	 By assumption, every root \(b\in\mathcal{R}(Q_i)\) satisfies 
	 \[ \overline{v}_{b,\d_i}|_{K(X)} = w_i. \]
	 Hence
	 \[ \left| \mathcal{S}_{w_i} (Q_i) \right| = \deg Q_i.  \]
	 Therefore, employing Theorem \ref{Thm S_f and j_w(f)} we obtain
	 \[ \left| \mathcal{S}_{w_i} (Q_{j+1}) \right| = t_i \, j_{w_i}(Q_{j+1}) \leq t_i \, n j_{w_i} (Q_i) = n \deg Q_i \leq \deg Q_{j+1},  \]
	 with at least one of the inequalities strict. It follows that \( \left| \mathcal{S}_{w_i} (Q_{j+1}) \right| < \deg Q_{j+1} \). Consequently, there exists a root
	 \[ a \in \mathcal{R}(Q_{j+1}) \setminus \mathcal{S}_{w_i}(Q_{j+1}).  \]
	 Thus
	 \[ \overline{v}_{a,\d_i}|_{K(X)} \neq w_i. \]
	 Since \( \left| \mathcal{S}_{w_i} (Q_i) \right| = \deg Q_i \), this implies that
	 \[ \overline{v}_{a,\d_i} \neq \overline{v}_{b,\d_i} \text{ for all } b\in\mathcal{R}(Q_i).  \]
	 Applying Proposition \ref{Prop pair of defn}, we obtain 
	 \[ \overline{v}(a-b) < \d_i \text{ for all } b\in\mathcal{R}(Q_i).  \]
	 \cite[Lemma 6.2]{MahboubMansourSpivakovsky2021} then yields that
	 \[ \overline{v} Q_i (a) < w Q_i.  \] 
	 On the other hand, we have
	 \[ \overline{v}_{a,\d_{j+1}}|_{K(X)} = w_{j+1} \]
	 from the given conditions. It follows from Theorem \ref{Thm key pols min pairs} that \( (a,\d_{j+1}) \) is a minimal pair of definition for \( (K(X)|K,w_{j+1}) \). Since \( \deg Q_i < \deg Q_{j+1} \), Theorem \ref{Thm w = v_Q} yields that
	 \[ w_{j+1} Q_i = \overline{v} Q_i (a). \]
	 Finally, since \( w Q_i = w_{j+1} Q_i \) by Proposition \ref{Prop wf = w_Q f}, we conclude that
	 \[ w Q_i = \overline{v} Q_i(a),  \]
	 contradicting our earlier observation. Thus \(Q_{j+1}\) is a regular limit key polynomial. We have thus proved the theorem.
\end{proof}

\pars We then obtain the following as an immediate consequence of Corollary \ref{Coro res = w dense in K^h}:

\begin{Lemma}
	Assume that \((K,v)\) is dense in the henselization \(K^h\). Then every arbitrary extension \( (K(X)|K,w) \) admits a regular complete sequence of key polynomials.
\end{Lemma}

\bibliographystyle{alpha}
\bibliography{references}

\end{document}